\documentclass[10pt,reqno]{amsart}
\usepackage{geometry}
\usepackage{amssymb}
\usepackage{graphicx}
\usepackage{amscd}
\usepackage{amsmath}
\usepackage{amsfonts}
\usepackage{mathrsfs}

\providecommand{\U}[1]{\protect\rule{.1in}{.1in}}
\newcommand{\R}{\mathbb{R}}

\providecommand{\U}[1]{\protect\rule{.1in}{.1in}}
\providecommand{\U}[1]{\protect\rule{.1in}{.1in}}

\newtheorem{example}{Example}[section]
\newtheorem{theorem}{Theorem}[section]
\newtheorem{defin}{Definition} [section]

\newtheorem{propo}{ Proposition}[section]

\newtheorem{coro}{Corollary} [section]

\numberwithin{equation}{section}

\begin{document}
\title[]{$\mathrm{L}^{2}_{\alpha} $-Solutions for Nonlinear Schrödinger Equations  Associated With The Weinstein Operator}
\author[Y. Bettaibi ]{}
\maketitle
\centerline{\bf Youssef Bettaibi}
\centerline{E-mail : youssef.bettaibi@yahoo.com}
\centerline{ University of Gabes, Faculty of Sciences of Gabes, LR17ES11 Mathematics }\centerline{and Applications, 6072, Gabes, Tunisia
}

\vskip 1cm

\begin{abstract}
In this paper, we study the Schrödinger equation associated with the Weinstein operators and we prove the existence and uniqueness of global solutions to Schrödinger-Weinstein equations in\\ $C\left(\left(-T_{\min }, T_{\max }\right) ; L_{\alpha}^{2}\left(\mathbb{R}^{d+1}_+\right)\right) \bigcap L_{l o c}^{q_{1}}\left(\left(-T_{\min }, T_{\max }\right) ; L_{\alpha}^{r_{1}}\left(\mathbb{R}^{d+1}_+\right)\right)$ \\ on initial condition in $L_{\alpha}^{2}\left(\mathbb{R}^{d+1}_+\right)$ 
\end{abstract}

{\bf 2010 AMS Classification : } {35Q55;42B35;35Q41;42B37;42B10}
\section{\textbf{Introduction}}
Weintein operators $\Delta_{W}^{\alpha,d}$ introduced by Weinsten in \cite{6} are parameterized differential operators on $\R^{d+1}_+$. Over the last few years, much attention has been paid to these operators
in various mathematical  directions and even physical for example.: The harmonic analysis associated with the Weinstein operator is studied by Ben Nahia and Ben Salem (\cite{1}\cite{2}). In particular the authors have introduced and studied the generalized Fourier transform associated with
the Weinstein operator. This transform is called the
Weinstein transform, and several authors have been interested in spaces related to this operator, in \cite{444} the authors introduced the Sobolev space associated with a Weinstein operator $H_{\mathcal{S}_{*}}^{s, \alpha}\left(\mathbb{R}_{+}^{d+1}\right)$ and investigated their properties, and in \cite{mj1} the authors  introduced the Sobolev refine the inequality between the homogeneous
Weinstein-Besov spaces $\dot{\mathcal{B}}_{p, q}^{\mathrm{s}_{, j} \beta}\left(\mathbb{R}_{+}^{d+1}\right)$ and many more, such as the homogenous Weinstein-Riesz spaces $\mathcal{R}_{\beta}^{-s}\left(L_{\beta}^{p}\left(\mathbb{R}_{+}^{d+1}\right)\right)$ and the generalized Lorentz spaces. B.Youssef \cite{bettaibi}  discuss an unique existence solutions  of Navier-Stokes equation associated with the Weinstein operator.\\

This work is devoted to define and to study the Schrödingen equations associated with the Weinstein
 operator (SW) in the whole $\mathbb{R}_{+}^{d+1}$ space of the form
\begin{equation*}\label{SW}\left( SW\right) \left\{\begin{array}{lll}
\partial_{t} u(t, x)-i \Delta_{W}^{\alpha,d} u(t, x) &=F(u(t, x)), \quad(t, x) \in I \times \mathbb{R}^{d+1}_+ \\
u_{\mid t=0} &=g \in L_{\alpha}^{2}\left(\mathbb{R}^{d+1}_+\right)
\end{array}\right.\end{equation*}where $d \geq 2, u$ is complex-valued function defined on $I \times \mathbb{R}^{d+1}_+$ and the nonlinearity $F \in C(\mathbb{C}, \mathbb{C})$ satisfies
\begin{align}\label{eq}\begin{array}{lll}
F(0)=0, \quad|F(u)-F(v)| \leq C\left(|u|^{p}+|v|^{p}\right)|u-v|, \quad p>0
\end{array}\end{align}
 We study a general theory for the above equation focusing on the following problems:\\
 *The dispersive phenomena and we prove
 the Strichartz estimates for this equation.\\
**Unique global existence of solutions for the (SW).\\

This paper is organized as follows : In section 2 we recall some elements
of harmonic analysis associated with the Weinstein operator, which will bee
needed in the sequel. Finally, Section 3 we establish the Strichartz estimates for the Weinsein-Schrödinger equation. Besides, we introduce a class of nonlinear Schrödinger equations associated with the Weinstein operators. In this regard, we study local and global well-posedness and scattering theory associated with these equations.
\section{\textbf{Harmonic analysis associated with the Weinstein-Laplace operator}}
\noindent\textbf{Notations.} In what follows, we need the following
notations:\newline$\bullet\;\mathbb{R}_{+}^{d+1}=\mathbb{R}^{d}\times\left]
0,\infty\right[  $.
\newline$\bullet\;x=(x_{1},...,x_{d},x_{d+1}%
)=(x^{\prime},x_{d+1})\in\mathbb{R}_{+}^{d+1}$
\newline$\bullet\; \vert x\vert=\sqrt{x_1^2+x_2^2+...+x_{d+1}^2}$.
\newline$\bullet\;\mathscr
C_{\ast}(\mathbb{R}^{d+1}),\;$the space of continuous functions on
$\mathbb{R}^{d+1}$, even with respect to the last variable.\newline%
$\bullet\;\mathscr C_{\ast,c}(\mathbb{R}^{d+1}),\;$the space of continuous
functions on $\mathbb{R}^{d+1}$ with compact support, even with respect to the
last variable.\newline$\bullet\;\mathscr C_{\ast}^{p}(\mathbb{R}^{d+1}),\;$the
space of functions of class $C^{p}$ on $\mathbb{R}^{d+1}$, even with respect
to the last variable.\newline$\bullet\;\mathscr E_{\ast}(\mathbb{R}^{d+1}%
),\;$the space of $C^{\infty}$-functions on $\mathbb{R}^{d+1}$, even with
respect to the last variable.\newline$\bullet\;\mathscr S_{\ast}%
(\mathbb{R}^{d+1}),\;$the Schwartz space of rapidly decreasing functions on
$\mathbb{R}^{d+1}$, even with respect to the last variable.\newline%
$\bullet\;\mathscr D_{\ast}(\mathbb{R}^{d+1}),\;$the space of $C^{\infty}%
$-functions on $\mathbb{R}^{d+1}$ which are of compact support, even with
respect to the last variable.\newline
\newline$\bullet\;L_{\alpha}%
^{p}(\mathbb{R}_{+}^{d+1}),$ $1\leq p\leq+\infty,\;$the space of measurable
functions on $\mathbb{R}_{+}^{d+1}$ such that
\[%
\begin{array}
[c]{lll}%
\Vert f\Vert_{\alpha,p} & = & \left[  \int_{\mathbb{R}_{+}^{d+1}}%
|f(x)|^{p}d\mu_{\alpha,d}(x)\right]  ^{\frac{1}{p}}<+\infty,\text{ if }1\leq
p<+\infty,\\
&  & \\
\Vert f\Vert_{\alpha,\infty} & = & \mathrm{ess}\underset{x\in\mathbb{R}%
	_{+}^{d+1}}{\sup}\left\vert f(x)\right\vert <+\infty,
\end{array}
\]
where $\mu_{\alpha,d}$ is the measure defined on $\mathbb{R}_{+}^{d+1}$ by
\begin{equation}
d\mu_{\alpha,d}(x)=\frac{x_{d+1}^{2\alpha+1}}{\left(  2\pi\right)  ^{\frac
		{d}{2}}2^{\alpha}\Gamma(\alpha+1)}dx, \label{1.2}%
\end{equation}
and $dx$ is the Lebesgue measure on $\mathbb{R}^{d+1}$. \newline%

In this section, we shall collect some results and definitions from the theory
of the harmonic analysis associated with the Weinstein operator developed in
\cite{bn}.\newline

The Weinstein operator $\Delta_{W}^{\alpha,d}$ is defined on $\mathbb{R}%
_{+}^{d+1}=\mathbb{R}^{d}\times\left]  0,\ +\infty\right[  ,$ by:
\begin{equation}
\label{1.1}\Delta_{W}^{\alpha,d}=\sum_{i=1}^{d+1}\frac{\partial^{2}}{\partial
	x_{i}^{2} }+\frac{2\alpha+1}{x_{d+1}}\frac{\partial}{\partial x_{d+1}}%
=\Delta_{d}+L_{\alpha},~~\alpha>-\frac{1}{2},
\end{equation}
where $\Delta_{d}$ is the Laplacian for the $d$ first variables and
$L_{\alpha}$ is the Bessel operator for the last variable defined on $\left]
0,\ +\infty\right[  $ by :
\[
L_{\alpha}u=\frac{\partial^{2}u}{\partial x_{d+1}^{2}}+\frac{2\alpha
	+1}{x_{d+1}}\frac{\partial u}{\partial x_{d+1}}=\frac{1}{x_{d+1}^{2\alpha+1}
}\frac{\partial}{\partial x_{d+1}}\left[  x_{d+1}^{2\alpha+1}\frac{\partial
	u}{\partial x_{d+1}}\right]  .
\]
The Weinstein operator $\Delta_{W}^{\alpha,d}$, mostly referred to as the
Laplace-Bessel differential operator is now known as an important operator in
analysis. The relevant harmonic analysis associated with the Bessel
differential operator$L_{\alpha}$ goes back to S. Bochner, J. Delsarte, B.M.
Levitan and has been studied by many other authors such as J. L\"{o}fstr\"{o}m
and J. peetre \cite{Lof}, I. Kipriyanov \cite{Kip}, K.
Trim\`{e}che \cite{trim}, I.A. Aliev and B. Youssef \cite{Aliev2}.

Let us begin by the following result, which gives the eigenfunction
$\Psi_{\lambda}^{\alpha,d}$ of the Weinstein operator $\Delta_{W}^{\alpha,d}.$

\begin{propo}
	For all $\lambda=\left(  \lambda_{1},\lambda_{2},...,\lambda_{d+1}\right)
	\in\mathbb{C}^{d+1}$, the system
	\begin{equation}
	\left\{
	\begin{array}{lll}%
	\frac{\partial^{2}u}{\partial x_{j}^{2}}\left(  x\right)  =-\lambda_{j}%
	^{2}u(x),&\text{ if }~~~1\leq j\leq d\\
	L_{\alpha}u\left(  x\right)  =-\lambda_{d+1}^{2}u\left(  x\right)  ,\\
	u\left(  0\right)  =1,\;\frac{\partial u}{\partial x_{d+1}}(0)=0\\ \frac{\partial u}{\partial x_{j}}(0)=-i\lambda_{j},~~ &if ~~~1\leq j\leq
	d.
	\end{array}
	\right.  \label{2.1}%
	\end{equation}
	has a unique solution $\Psi^{\alpha}_{d} ({\lambda},.)\;$given by :
	\begin{equation}
	\forall z\mathbb{\in C}^{d+1},\;\Psi_{\lambda}^{\alpha,d}\left(  z\right)
	=e^{-i\left\langle z^{\prime}\text{,}\lambda^{\prime}\right\rangle }j_{\alpha
	}(\lambda_{d+1}z_{d+1}), \label{2.2}%
	\end{equation}
	where $z=(z^{\prime},x_{d+1}),\;z^{\prime}=\left(  z_{1},z_{2},...,z_{d}%
	\right)  $ and $j_{\alpha}$ is the normalized Bessel function of index
	$\alpha,\;$defined by
	$$
	\forall\xi\mathbb{\in C},\;j_{\alpha}(\xi)=\Gamma(\alpha+1)\underset{n=0}%
	{\sum^{\infty}}\frac{(-1)^{n}}{n!\Gamma(n+\alpha+1)}(\frac{\xi}{2})^{2n}.
	$$	
\end{propo}
\begin{propo}
	~~\newline i) For all $\lambda,\;z\in\mathbb{C}^{d+1}$ and $t\in\mathbb{R}$,
	we have
	\[
\Psi^{\alpha}_{d}\left(  \lambda,0\right)  =1,\;\Psi^{\alpha}_{d}\left(
	\lambda,z\right)  =\Psi^{\alpha}_{d}\left(  z,\lambda\right)  \;\text{and
	}\Psi^{\alpha}_{d}\left(  \lambda,tz\right)  =\Psi^{\alpha}_{d}\left(
	t\lambda,z\right)  .
	\]
	ii) For all $\nu\in\mathbb{N}^{d+1},\;x\in\mathbb{R}_{+}^{d+1}$ and
	$z\in\mathbb{C}^{d+1}$, we have
	\begin{equation}
	\label{2.3}|D_{z}^{\nu}\Psi^{\alpha}_{d}(x,z)|\leq\|x\|^{|\nu|}%
	\,\exp(\|x\|\,\|\operatorname{Im}z\|),
	\end{equation}
	where $D_{z}^{\nu}=\frac{\partial^{\nu}}{\partial z_{1}^{\nu_{1}}...\partial
		z_{d+1}^{\nu_{d+1}}}$ and $|\nu|=\nu_{1}+...+\nu_{d+1}.$ In particular
	\begin{equation}
	\label{2.4}\forall x,y\in\mathbb{R}_{+}^{d+1},\;|\Psi^{\alpha}_{d}%
	(x,y)|\leq1.
	\end{equation}
	
\end{propo}
\begin{defin}
	The Weinstein transform is given for $f\in L_{\alpha}^{1}(\mathbb{R}_{+}%
	^{d+1})$ by
	\begin{equation}
	\forall\lambda\in\mathbb{R}_{+}^{d+1},\;\mathscr F_{W}^{\alpha,d}%
	(f)(\lambda)=\int_{\mathbb{R}_{+}^{d+1}}f(x)\Psi^{\alpha}_{d}(x,\lambda
	)d\mu_{\alpha,d}(x). \label{2.10}%
	\end{equation}
	where $\mu_{\alpha,d}$ is the measure on $\mathbb{R}_{+}^{d+1}$ given by the
	relation (\ref{1.2}).
\end{defin}

The Weinstein tansform, referred to as the Fourier-Bessel transform, has been
investigated by I. Kipriyanov \cite{Kip}, I.A. Aliev \cite{Aliev1} and others
\newline( see \cite{Aliev3},\cite{bn},\cite{ bn1} and \cite{bet} ).

Using the properties of the classical Fourier transform on $\mathbb{R}^{d}$
and of the Bessel transform, one can easily see the following relation, which
will play an important role in the sequel.

\begin{example}\label{ex}
1)	Let $E_{s},\;s>0,$ be the function defined by
	\[
	\forall x\in\mathbb{R}^{d+1},\;E_{s}\left(  x\right)  =e^{-s\left\vert
		x\right\vert ^{2}}.
	\]
	Then the Weinstein transform $\mathscr F_{W}^{\alpha,d}$ of $E_{s}$ is given
	by :
	\begin{equation}\label{2.11}
	\forall\lambda\in\mathbb{R}_{+}^{d+1},\;\mathscr F_{W}^{\alpha,d}%
	(E_{s})(\lambda)=\frac{1}{\left(  2s\right)  ^{\alpha+\frac{d}{2}+1}}%
	e^{-\frac{\left\vert \lambda\right\vert ^{2}}{4s}}. %
	\end{equation}
From (\ref{2.11}), we have for any positive $s$
\begin{equation}
,\forall\lambda\in\mathbb{R}_{+}^{d+1}\;\mathscr F_{W}^{\alpha,d}
(E_{s})(\lambda)=\frac{1}{\left(  2s\right)  ^{\alpha+\frac{d}{2}+1}}%
e^{-\frac{\left\vert \lambda\right\vert ^{2}}{4s}}. %
\end{equation}
Let us observe that for $z \in \mathbb{C}, z=|z| e^{i \theta}$ with positive real part, by taking the branch $z^{\frac{1}{2}}=|z|^{\frac{1}{2}} e^{i \frac{\theta}{2}}$, the two functions
\end{example}
$$
z \mapsto \mathscr F_{W}^{\alpha,d}
(E_{z}) \quad \text { and } \quad z \mapsto \frac{1}{\left(  2z\right)  ^{\alpha+\frac{d}{2}+1}}%
e^{-\frac{\left\vert .\right\vert ^{2}}{4z}}
$$
are holomorphic on the domain $\operatorname{Re} z>0$. As they coincide on the real axis, they coincide in the whole domain. Now, if $t$ positive, considering a sequence of $z_{n}$ with positive real part which tends to $i t$, we get, as the Weinstein transform is continuous on tempered distributions that

\begin{align}
\forall\lambda\in\mathbb{R}_{+}^{d+1},\quad  \mathscr F_{W}^{\alpha,d}(E_{is})(\lambda)&=
\frac{1}{\left(  2is\right)  ^{\alpha+\frac{d}{2}+1}}%
e^{i\frac{\left\vert \lambda\right\vert ^{2}}{4s}}\\&=\frac{1}{\left(  2s\right)  ^{\alpha+\frac{d}{2}+1}}e^{-i\frac{\pi}{2}(\alpha+\frac{d}{2}+1)}
e^{i\frac{\left\vert \lambda\right\vert ^{2}}{4s}}
\end{align}
If $s\in \R $ with $s\neq 0$, we have 

\begin{align}
\forall\lambda\in\mathbb{R}_{+}^{d+1},\quad  \mathscr F_{W}^{\alpha,d}(E_{is})(\lambda)=\frac{1}{\left\vert  2s\right\vert  ^{\alpha+\frac{d}{2}+1}}e^{-i\frac{\pi}{2}(\alpha+\frac{d}{2}+1) \operatorname{sgn}(s) }
e^{i\frac{\left\vert \lambda\right\vert ^{2}}{4s}}
\end{align}

\noindent Some basic properties of the transform $\mathscr F_{W}^{\alpha,d}$
are summarized in the following results. For the proofs, we refer to \cite{bn,
	bn1}.

\begin{propo}\label{l1}
	(see \cite{bn, bn1})~~\newline i) For all $f\in L_{\alpha}^{1}(\mathbb{R}%
	_{+}^{d+1})$, we have
	\begin{equation}
	\Vert\mathscr F_{W}^{\alpha,d}(f)\Vert_{\alpha,\infty}\leq\Vert f\Vert
	_{\alpha,1}.\label{2.12}%
	\end{equation}
	ii) For all $f\in L_{\alpha}^{1}(\mathbb{R}%
	_{+}^{d+1})$ and $\Delta_{W}^{\alpha,d}f\in L_{\alpha}^{1}(\mathbb{R}%
	_{+}^{d+1})$  we have
	\begin{align} \label{xx}
	\mathscr F_{W}^{\alpha,d}(\Delta_{W}^{\alpha,d}f)(x)=-\vert x\vert^2\mathscr F_{W}^{\alpha,d}(f)(x)
	\end{align}
\end{propo}
\begin{theorem}
	(see \cite{bn, bn1})~~\newline i\textbf{) }The Weinstein transform $\mathscr
	F_{W}^{\alpha,d}$ is a topological isomorphism \ from $\mathscr S_{\ast
	}(\mathbb{R}^{d+1})$ onto itself and from $\mathscr D_{\ast}(\mathbb{R}%
	^{d+1})$ onto $\mathcal{H}_{\ast}(\mathbb{C}^{d+1}\mathbb{)}$.\newline ii) Let
	$f\;\in\mathscr S_{\ast}(\mathbb{R}^{d+1})$. The inverse transform $\left(
	\mathscr F_{W}^{\alpha,d}\right)  ^{-1}\;$is given by
	\begin{equation}
	\forall x\in\mathbb{R}_{+}^{d+1},\;\left(  \mathscr F_{W}^{\alpha,d}\right)
	^{-1}(f)(x)=\mathscr F_{W}^{\alpha,d}(f)\left(  -x\right)  .\label{2.16}%
	\end{equation}
	iii) Let $f\in L_{\alpha}^{1}(\mathbb{R}_{+}^{d+1})$. If $\mathscr
	F_{W}^{\alpha,d}(f)\in L_{\alpha}^{1}(\mathbb{R}_{+}^{d+1}),$ then we have
	\begin{equation}
	f(x)=\int_{\mathbb{R}_{+}^{d+1}}\mathscr F_{W}^{\alpha,d}(f)\left(  y\right)
\Psi^{\alpha}_{d}(-x,y)d\mu_{\alpha,d}(y),\;a.e\;x\in\mathbb{R}_{+}%
	^{d+1}.\label{2.17}%
	\end{equation}
	
\end{theorem}

\begin{theorem}
	(see \cite{bn, bn1})~~\newline i) For all $f,g\in\mathscr S_{\ast}%
	(\mathbb{R}^{d+1}),$ we have the following Parseval formula
	\begin{equation}
	\int_{\mathbb{R}_{+}^{d+1}}f(x)\overline{g(x)}d\mu_{\alpha,d}(x)=\int
	_{\mathbb{R}_{+}^{d+1}}\mathscr F_{W}^{\alpha,d}(f)(\lambda)\overline{\mathscr
		F_{W}^{\alpha,d}(g)(\lambda)}d\mu_{\alpha,d}(\lambda).\label{2.18}%
	\end{equation}
	ii) ( Plancherel formula ). \newline For all $f\in\mathscr S_{\ast}%
	(\mathbb{R}^{d+1}),$ we have :
	\begin{equation}
	\int_{\mathbb{R}_{+}^{d+1}}\left\vert f(x)\right\vert ^{2}d\mu_{\alpha
		,d}(x)=\int_{\mathbb{R}_{+}^{d+1}}\left\vert \mathscr F_{W}^{\alpha
		,d}(f)(\lambda)\right\vert ^{2}d\mu_{\alpha,d}(\lambda).\label{2.19}%
	\end{equation}
	iii) ( Plancherel Theorem ) :\newline The transform $\mathscr F_{W}^{\alpha
		,d}$ extends uniquely to an isometric isomorphism on $L_{\alpha}%
	^{2}(\mathbb{R}_{+}^{d+1}).$
\end{theorem}
\begin{defin}
	The translation operator $T_{x},\;$ $x\in\mathbb{R}_{+}^{d+1}$, associated
	with the Weinstein operator, is defined on $C_{\ast
	}(\mathbb{R}^{d+1}),$ for all $y\in\mathbb{R}_{+}^{d+1}$, by :
	\[
	T_{x}f\left(  y\right)  =\frac{a_{\alpha}}{2}\int_{0}^{\pi}f\left(  x^{\prime
	}+y^{\prime},\;\sqrt{x_{d+1}^{2}+y_{d+1}^{2}+2x_{d+1}y_{d+1}\cos\theta
	}\right)  \left(  \sin\theta\right)  ^{2\alpha}d\theta,
	\]
	where $x^{\prime}+y^{\prime}=\left(  x_{1}+y_{1},...,x_{d}+y_{d}\right)  $ and
	$a_{\alpha}$ is the constant given by (\ref{2.6}).
\end{defin}
The following propo summarizes some properties of the Weinstein
translation operator.

\begin{propo}
	(see \cite{bn, bn1})~~\newline i) For $f\in C_{\ast}(\mathbb{R}^{d+1})$, we
	have
	\[
	\forall x,\;y\in\mathbb{R}_{+}^{d+1},\;T_{x}f\left(  y\right)  =T_{y}f\left(
	x\right)  \text{ and }T_{0}f=f.
	\]
	ii) For all $f\in\mathscr E_{\ast}(\mathbb{R}^{d+1})$ and $y\in\mathbb{R}%
	_{+}^{d+1}$, the function $x\mapsto T_{x}f\left(  y\right)  $ belongs to
	$\mathscr E_{\ast}(\mathbb{R}^{d+1}).$\newline iii) We have
	\[
	\forall x\in\mathbb{R}_{+}^{d+1},\;\Delta_{W}^{\alpha,d}\circ T_{x}=T_{x}%
	\circ\Delta_{W}^{\alpha,d}.
	\]
	iv) Let $f\in L_{\alpha}^{p}(\mathbb{R}_{+}^{d+1}),\;1\leq p\leq+\infty$ and
	$x\in\mathbb{R}_{+}^{d+1}$. Then $T_{x}f$ belongs to $L_{\alpha}%
	^{p}(\mathbb{R}_{+}^{d+1})$ and we have
	\[
	\Vert T_{x}f\Vert_{\alpha,p}\leq\Vert f\Vert_{\alpha,p}.
	\]
	v) The function $\Psi^{\alpha}_{d}\left(  .,\lambda\right)  ,$ $\lambda
	\in\mathbb{C}^{d+1},\;$ satisfies on $\mathbb{R}_{+}^{d+1}$ the following
	product formula:
	\begin{equation}
	\forall y\in\mathbb{R}_{+}^{d+1},\;\Psi^{\alpha}_{d}\left(  x,\lambda\right)
\Psi^{\alpha}_{d}\left(  y,\lambda\right)  =T_{x}\left[ \Psi^{\alpha}_{d}\left(  .,\lambda\right)  \right]  \left(  y\right)  .\label{2.23}%
	\end{equation}
	\newline vi) Let $f\in L_{\alpha}^{p}(\mathbb{R}_{+}^{d+1}),\;p=1$ or $2$ and
	$x\in\mathbb{R}_{+}^{d+1}$, we have
	\begin{equation}
	\forall y\in\mathbb{R}_{+}^{d+1},\;\mathscr F_{W}^{\alpha,d}\left(
	T_{x}f\right)  \left(  y\right)  =\Psi^{\alpha}_{d}\left(  x,y\right)
	\mathscr F_{W}^{\alpha,d}\left(  f\right)  \left(  y\right)  .\label{2.24}%
	\end{equation}
	vii) The space $\mathscr S_{\ast}(\mathbb{R}^{d+1})$ is invariant under the
	operators $T_{x},\;x\in\mathbb{R}_{+}^{d+1}.$\newline
\end{propo}

\begin{defin}
	The Weinstein convolution product of $f,g\in\mathscr C_{\ast}(\mathbb{R}%
	^{d+1})$ is given by:
	\begin{equation}
	\forall x\in\mathbb{R}_{+}^{d+1},\;f\ast_{W}g\left(  x\right)  =\int
	_{\mathbb{R}_{+}^{d+1}}T_{x}f\left(  y\right)  g\left(  y\right)  d\mu
	_{\alpha,d}(y).\label{2.25}%
	\end{equation}
\end{defin}
  
\begin{propo}\label{pro5}
	(see \cite{bn, bn1})~~\newline i) Let $p,q,r\in\left[  1,\;+\infty\right]  $
	such that $\frac{1}{p}+\frac{1}{q}-\frac{1}{r}=1.$ Then for all $f\in
	L_{\alpha}^{p}(\mathbb{R}_{+}^{d+1})$ and$\;g\in L_{\alpha}^{q}(\mathbb{R}%
	_{+}^{d+1}),$ the function $f\ast_{W}g$ $\in L_{\alpha}^{r}(\mathbb{R}%
	_{+}^{d+1})$ and we have
	\begin{equation}
	\Vert f\ast_{W}g\Vert_{\alpha,r}\leq\Vert f\Vert_{\alpha,p}\Vert
	g\Vert_{\alpha,q}.\label{2.26}%
	\end{equation}
	ii) For all $f,g\in L_{\alpha}^{1}(\mathbb{R}_{+}^{d+1}),\;\left(
	resp.\;\mathscr S_{\ast}(\mathbb{R}^{d+1})\right)  ,\;f\ast_{W}g$ $\in
	L_{\alpha}^{1}(\mathbb{R}_{+}^{d+1})$ $\left(  resp.\;\mathscr S_{\ast
	}(\mathbb{R}^{d+1})\right)  \;$and we have
	\begin{equation}
	\mathscr F_{W}^{\alpha,d}(f\ast_{W}g)=\mathscr F_{W}^{\alpha,d}(f)\mathscr
	F_{W}^{\alpha,d}(g).\label{2.27}%
	\end{equation}
	\begin{example}\label{ex3}
 Let $ t>0$, then for all $x, y \in \mathbb{R}_{+}^{d+1}$, we have
$$
T_{x}\left(e^{-t|.|}\right)(y)=\frac{1}{(2 t)^{\alpha+\frac{d}{2}+1}} e^{-t(|x|^{2}+| y|^{2})} \Lambda_{\alpha, d}\left(x,- {2i t}{y}\right)
$$
then, for all $f \in L_{\alpha}^{1}(\mathbb{R}%
_{+}^{d+1})$, we have 
$$\left(e^{-t|.|}\ast_{W}f\right)(x)=\frac{1}{(2 t)^{\frac{d+2\alpha+2}{2}}} e^{{-t|x|^{2}} }\mathscr F_{W}^{\alpha,d}(e^{-t|x|^{2}}f)(-it{x}).
$$
	\end{example}
\end{propo}
\section{\textbf{Schrödingen equation associated with the Weinstein operator}}
Notations. For any interval $I$ of $\mathbb{R}$ (bounded or unbounded) and a Banach space $X$, we define the mixed space-time $L^{q}(I ; X)$ Banach space of (classes of) measurable functions $u: I \rightarrow X$ such that $\|u\|_{L^{q}(I ; X)}<\infty$, with
$$
\begin{aligned}
\|u\|_{L^{q}(I ; X)} &=\left(\int_{I}\|u(t, \cdot)\|_{X}^{q} d t\right)^{\frac{1}{q}}, \quad \text { if } 1 \leq q<\infty \\
\|u\|_{L^{\infty}(I ; X)} &=e s s \sup _{t \in I}\|u(t, \cdot)\|_{X}
\end{aligned}
$$
$C(\bar{I} ; X)$ the space of continuous functions $\bar{I} \rightarrow X$. When $I$ is bounded, $C(\bar{I} ; X)$ is a Banach space with the norm of $L^{\infty}(I, X)$.\\
$C_{c}\left(I, \mathcal{S}\left(\mathbb{R}^{d}\right)\right)$ is the space of continuous functions from $I$ into $\mathcal{S}\left(\mathbb{R}^{d}\right)$ compactly supported in $I$, equipped with the topology of uniform convergence on the compact subintervals of $I$.
\subsection{Dispersion phenomena}
 
 \begin{defin}
  We say that the exponent pair $(q, r)$ is $\frac{d+2\alpha+2}{2}$-admissible if $q, r \geq 2,\left(q, r, \frac{d+2\alpha+2}{2}\right) \neq(2, \infty, 1)$ and
 \begin{equation}\label{sh}
 \frac{1}{q}+\frac{d+2\alpha+2}{2 r} \leq \frac{d+2\alpha+2}{4}
\end{equation}
 If equality holds in \eqref{sh}, we say that $(q, r)$ is sharp $\frac{d+2\alpha+2}{2}-a d m i s s i b l e$, otherwise we say that $(q, r)$ is nonsharp $\frac{d+2\alpha+2}{2}$-admissible. Note in particular that when $d+2\alpha>0$ the endpoint
 $$
 P=\left(2, \frac{2d+4 \alpha +4}{d+2 \alpha}\right)
 $$
 is sharp $\frac{d+2\alpha+2}{2}$-admissible.	
\end{defin}
\begin{theorem}\label{TH3}
 Let $(U(t))_{t \in \mathbb{R}}$ be a bounded family of continuous operators on $L_{\alpha}^{1}(\mathbb{R}_{+}^{d+1})$ such that, we have
$$
\left\|U(t) U^{*}\left(t^{\prime}\right) f\right\|_{\alpha,\infty} \leq \frac{C}{\left|t-t^{\prime}\right|^{\frac{d+2\alpha+2}{2}}}\|f\|_{\alpha,1}
$$
Then, the estimates

\begin{align}
\left\|U(t) g\right\|_{L^{q}\left(\mathbb{R} ; L_{\alpha}^{r}\left(\mathbb{R}^{d+1}_+\right)\right)} & \leq C\left\|g\right\|_{\alpha,2} \label{eq11}\\
\left\|\int_{\mathbb{R}} U^{*}(t) f(t, \cdot) d t\right\|_{\alpha,2} & \leq C\|f\|_{L^{q^{\prime}}\left(\mathbb{R} ; L_{\alpha}^{r^{\prime}}\left(\mathbb{R}^{d+1}_+\right)\right)}\label{eq12}
\end{align}

hold for any sharp $\frac{d+2\alpha+2}{2}-a d m i s s i b l e$ exponent $(q, r)$, where $q^{\prime}, r^{\prime}$ are the conjugate exponents of $q$ and $r$ and $U^{*}$ is the adjoint operator of $U$.

Moreover, for any sharp $\frac{d+2\alpha+2}{2}$-admissible exponent pairs $(q, r)$ and $\left(q_{1}, r_{1}\right)$ we have
\begin{align}\label{eq13}
\left\|\int_{\mathbb{R}} U(t) U^{*}\left(t^{\prime}\right) f\left(t^{\prime}, \cdot\right) d t^{\prime}\right\|_{L^{q}\left(\mathbb{R} ; L_{\alpha}^{r}\left(\mathbb{R}^{d+1}_+\right)\right)} \leq C\|f\|_{L^{q_{1}^{\prime}}\left(\mathbb{R} ; L_{\alpha}^{r_{1}^{\prime}}\left(\mathbb{R}^{d+1}_+\right)\right)}
\end{align}
Furthermore, if
$$
\left\|U(s) U^{*}(t) f\right\|_{\alpha,\infty} \leq \frac{C}{(1+|t-s|)^{\frac{d+ 2\alpha+2}{2}}}\|f\|_{\alpha,1}
$$
then $\eqref{eq11},\eqref{eq12}$ and $\eqref{eq13}$ hold for all $\frac{d+2\alpha+2}{2}-a d m i s s i b l e(q, r)$ and $\left(q_{1}, r_{1}\right)$.
\end{theorem}
\begin{proof}
	The proof of the theorem  uses the same idea as in \cite{8}.
	\end{proof}
\subsection{\textbf{ { Strichartz-type Schrödinger-Weinstein estimates }}}
 In this We consider in the rest of this article that the incompressible Weinstein-Schrödingen system is given by:
\begin{equation}\label{11}\left( SW\right) \left\{\begin{array}{lll}
\partial_{t} u(t,x)-i\Delta_{W}^{\alpha, d} u(t,x)=F(t,x), &\quad \text { in } \quad \mathbb{R} \times \mathbb{R}_{+}^{d+1}\\
u(0,x)=g(x) &\quad \text { in } \quad  \mathbb{R}_{+}^{d+1}
\end{array}\right.\end{equation}
Moreover, under the same conditions, Duhamel's formula implies
$$
u(t, x)=\mathcal{I}_{\alpha}(t) g(x)+\int_{0}^{t} \mathcal{I}_{\alpha}(t-s) F(s, x) d s, \quad(t, x) \in\R\times \mathbb{R}^{d+1}_+
$$
where $\mathcal{I}_{\alpha}(t)$ is the unitary operator defined by
\begin{align}\label{u}
\mathcal{I}_{\alpha}(t)u=\left(\mathcal{F}_{W}^{\alpha, d}\right)^{-1}\left(e^{-i t|\xi|^{2}} \mathcal{F}_{W}^{\alpha, d}{u}\right)
\end{align}
follows from Examples \ref{ex} and \ref{ex3} we have 
 \begin{align}\label{u2}
 \mathcal{I}_{\alpha}(t)u=&\frac{1}{\left(  2t\right)  ^{\alpha+\frac{d}{2}+1}}e^{-i\frac{\pi}{2}(\alpha+\frac{d}{2}+1)}
 e^{i\frac{\left\vert \lambda\right\vert ^{2}}{4t}} *_{W} u\\
 =&\frac{1}{(2t)^{\frac{d +2\alpha+2}{2}}} e^{-i(d+2\alpha+2) \frac{\pi}{4}  t} e^{i \frac{\|\cdot\|^{2}}{4 t}}\left[\mathcal{F}_{D}\left(e^{i \frac{\|\cdot\|^{2}}{4 t}} v\right)\right]\left(\frac{\cdot}{2 t}\right)
 \end{align}
\begin{theorem}\label{th4} Suppose that $d \geq 1$ and that $(q, r)$ and $\left(q_{1}, r_{1}\right)$ are $\frac{d+\alpha+3}{2}$-admissible pairs. If $u$ is a solution to the problem
	$$
	\left\{\begin{aligned}
	\partial_{t} u-i \triangle_{\alpha} u &=F(t, x),(t, x) \in[0, T] \times \mathbb{R}^{d+1}_+ \\
	u_{\mid t=0} &=g
	\end{aligned}\right.
	$$
	for some data, $g, F$ and time $0<T<\infty$, then
	$$
	\begin{aligned}
	&\|u\|_{L^{q}\left([0, T] ; L_{\alpha}^{r}\left(\mathbb{R}^{d+1}_+\right)\right)}+\|u\|_{C\left([0, T] ; L_{\alpha}^{2}\left(\mathbb{R}^{d+1}_+\right)\right)} \\
	&\quad \leq C\left(\|g\|_{\alpha, 2}+\|F\|_{L^{q_{1}^{\prime}}\left([0, T] ; L_{\alpha}^{r_{1}^{\prime}}\left(\mathbb{R}^{d+1}_+\right)\right)}\right)
	\end{aligned}
	$$
	Conversely, if the above estimate holds for all $g, F, T$, then $(q, r)$ and $\left(q_{1}, r_{1}\right) m u s t$ be $\frac{d+4 \alpha+2}{2}$-admissible.
\end{theorem}
\begin{proof}
We will prove the sufficient condition first. Indeed we assume that $(q, r)$ satisfy the condition of the theorem, and that $u$ is a solution of $(SW)$. \\
According to Plancherel formula and the equation\eqref{u}  we have
\begin{equation}\label{eq1}
\left\|\mathcal{I}_{\alpha}(t) g\right\|_{\alpha,2}=\|g\|_{\alpha,2}
\end{equation}
The estimate
\begin{equation}\label{eq2}
\left\|\mathcal{I}_{\alpha}(t-s) g\right\|_{\alpha,\infty} \leq \frac{C}{|t-s|^{\frac{d +2\alpha+2}{2}}}\|g\|_{\alpha,1}
\end{equation}
follows from Propositions \ref{pro5} and the equation \eqref{u2}. Below, we note by $\Phi_{\alpha}$ the operator defined by
$$
\Phi_{\alpha}(F)(t, x):=\int_{0}^{t} \mathcal{I}_{\alpha}(t-s) F(s, x) d s
$$
 Replacing the $C\left([0, T] ; L_{\alpha}^{2}\left(\mathbb{R}^{d+1}_+\right)\right)$ norm in the above by the $L^{\infty}\left([0, T] ; L_{\alpha}^{2}\left(\mathbb{R}^{d+1}_+\right)\right)$ norm, the all estimates will follow from Theorem \ref{TH3}.

We now address the question of continuity in $L_{\alpha}^{2} .$ The continuity of $\mathcal{I}_{\alpha}(\cdot) g$ follows from Plancherel formula. To show that the quantity $\Phi_{\alpha}(F)$ is continuous in $L_{\alpha}^{2}\left(\mathbb{R}^{d+1}_+\right)$, one can use the identity
$$
\Phi_{\alpha}(F)(t+\varepsilon)=\mathcal{I}_{\alpha}(\varepsilon)\left[\Phi_{\alpha}(F)(t)+\Phi_{\alpha}\left(1_{[t, t+\varepsilon]} F\right)(t)\right]
$$
the continuity of $\mathcal{I}_{\alpha}(\varepsilon)$ as an operator on $L_{\alpha}^{2}\left(\mathbb{R}^{d+1}_+\right)$, and the fact that
$$
\left\|1_{[t, t+\varepsilon]} F\right\|_{L^{q_{1}^{\prime}}\left([0, T] ; L_{\alpha}^{r_{1}^{\prime}}\left(\mathbb{R}^{d+1}_+\right)\right)} \rightarrow 0 \quad \text { as } \quad \varepsilon \rightarrow 0
$$
We finish the proof of necessity as in \cite{8}.
\end{proof}
\begin{coro}\label{co3}
Let $I$ be an interval of $\mathbb{R}$. If $(q, r)$ and $\left(q_{1}, r_{1}\right)$ are $\frac{d+2\alpha+2}{2}$-admissible pairs, then there exits a constant $C$ such that
\begin{align*}
\left\|\mathcal{I}_{\alpha}(\cdot) g\right\|_{L^{q}\left(\mathbb{R} ; L_{\alpha}^{r}\left(\mathbb{R}^{d}\right)\right)} \leq& C\|g\|_{\alpha,2}\\\left\|\Phi_{\alpha}(F)\right\|_{L^{q}\left(I ; L_{\alpha}^{r}\left(\mathbb{R}^{d}\right)\right)} \leq& C\|F\|_{\left.L^{q_{1}^{\prime}}\left(I ; L_{\alpha}\right.^{r_{1}^{\prime}}\left(\mathbb{R}^{d}\right)\right)}
\end{align*}
\end{coro}

\begin{propo}
	If $p \in[2, \infty]$ and $t \neq 0$, then $\mathcal{I}_{\alpha}(t)$ maps $L_{\alpha}^{p^{\prime}}\left(\mathbb{R}^{d+1}_+\right)$ continuously to $L_{\alpha}^{p}\left(\mathbb{R}^{d+1}_+\right)$ and
	\begin{align}\label{eq3}
	\left\|\mathcal{I}_{\alpha}(t) g\right\|_{\alpha,p} \leq \frac{1}{\left(2|t|\right)^{(d+2\alpha+2)\left(\frac{1}{2}-\frac{1}{p}\right)}}\|g\|_{{\alpha},{p^{\prime}}}
	\end{align}
\end{propo}
\begin{proof}
It follows from \eqref{eq1} and \eqref{eq2} that
	$$
	\left\|\mathcal{I}_{\alpha}(t) g\right\|_{{\alpha},{\infty}} \leq \frac{1}{|2t|^{\frac{d+2\alpha+2}{2}}}\|g\|_{{\alpha},{1}} \quad \text { and } \quad\left\|\mathcal{I}_{\alpha}(t) g\right\|_{{\alpha},{2}}=\|g\|_{{\alpha},{2}}
	$$
	The general case is obtained by interpolation between the cases $p=2$ and $p=\infty$
	\end{proof}
\begin{propo}
	Let $I$ be an interval of $\mathbb{R}$ (bounded or not). Assume $2<r<\frac{2 d+8 \alpha+6}{d+4 \alpha-1}$ and let $\left(q_{1}, r_{1}\right) \in(1, \infty)^{2}$ satisfy
	$$
	\frac{1}{q_{1}}+\frac{1}{r_{1}}=(d+2\alpha+2)\left(\frac{1}{2}-\frac{1}{r}\right)
	$$
	Then $\Phi_{\alpha}(F) \in L^{q_{1}}\left(I ; L_{\alpha}^{r}\left(\mathbb{R}^{d+1}_+\right)\right)$ for every $F \in L^{r_{1}^{\prime}}\left(I ; L_{\alpha}^{r^{\prime}}\left(\mathbb{R}^{d+1}_+\right)\right) .$ Moreover, there exits a constant $C$ independent on $I$ such that

	\begin{align}\label{eq4}
	&\left\|\Phi_{\alpha}(F)\right\|_{L^{q_{1}}\left(I ; L_{\alpha}^{r}\left(\mathbb{R}^{d+1}_+\right)\right)} \leq C\|F\|_{L^{r_{1}^{\prime}} \left(I ; L_{\alpha}^{r^{\prime}}\left(\mathbb{R}^{d+1}_+\right)\right)} \\
	&\qquad \text { for every } \quad F \in L^{r_{1}^{\prime}}\left(I ; L_{\alpha}^{r^{\prime}}\left(\mathbb{R}^{d+1}_+\right)\right)\nonumber
	\end{align}

\end{propo}
\begin{proof}
	Proof. By density, we need to prove \eqref{eq4} for $F \in C_{c}\left(I ; \mathcal{S}_*\left(\mathbb{R}^{d+1}\right)\right) .$ It follows from \eqref{eq3} that
	$$
	\left\|\Phi_{\alpha}(F)(t, \cdot)\right\|_{{\alpha},{r}} \leq \int_{0}^{t} \frac{1}{\left((2|t-s|)^{4 \alpha+d+3}\right)^{\left(\frac{1}{2}-\frac{1}{r}\right)}}\|F(s, \cdot)\|_{{\alpha},{r^{\prime}}} d s
	$$
	and so \eqref{eq4} is an immediate consequence of Hardy-Littlewood-Sobolev inequality.
	\end{proof}
\subsection{$\mathrm{L}^{2}_{\alpha}-\text {Solutions for nonlinear Weinstein-Schrödinger equations }$}
In this subsection, Strichartz estimates are a powerful tool to prove local and global well-posedness results for the nonlinear Weinstein-Schrödingen equations (SW), for this we recall the definition of well-posedness.
\begin{defin} We say that the problem (SW) is locally well-posed in $L_{\alpha}^{2}\left(\mathbb{R}^{d+1}_+\right)$ if, for every $g$ in $L_{\alpha}^{2}\left(\mathbb{R}^{d+1}_+\right)$, one can find time $T>0$ and a unique solution $u \in C\left([-T, T], L_{\alpha}^{2}\left(\mathbb{R}^{d+1}_+\right)\right) \cap X$ to (SW) which depends continuously on the data, with $X$ some additional Banach space. The equation is globally well-posed if these properties hold with $T=\infty$.
\end{defin}
\begin{theorem}\label{th5}
If $p \in\left(0, \frac{4}{d+2\alpha+2}\right]$, then for every $g \in L_{\alpha}^{2}\left(\mathbb{R}^{d+1}_+\right)$, there exist $T_{\max }, T_{\min } \in(0, \infty]$ and a unique, maximal solution $u$ of $(SW)$ belonging to
	$$
	C\left(\left(-T_{\min }, T_{\max }\right) ; L_{\alpha}^{2}\left(\mathbb{R}^{d+1}_+\right)\right) \bigcap L_{l o c}^{q}\left(\left(-T_{\min }, T_{\max }\right) ; L_{\alpha}^{r}\left(\mathbb{R}^{d+1}_+\right)\right)
	$$
	for every sharp $\frac{d+2\alpha+2}{2}$-admissible pair $(q, r)$. Moreover, the following properties hold:\\
 There exists $\delta_{0}>0$ such that if $g \in L_{\alpha}^{2}\left(\mathbb{R}^{d+1}_+\right)$ satisfies $\|g\|_{L_{\alpha}^{2}\left(\mathbb{R}^{d+1}_+\right)} \leq \delta_{0}$ then the corresponding maximal $L_{\alpha}^{2}$-solution is global, i.e., $T_{\max }=\hat{T}_{\min }=\infty$.
	Moreover, u belongs to $L^{q}\left(\mathbb{R} ; L_{\alpha}^{r}\left(\mathbb{R}^{d+1}_+\right)\right)$ for every sharp $\frac{d+2\alpha+2}{2}$-admissible pair $(q, r)$.
\end{theorem}
\begin{proof}
	Proof. We proceed in three steps.\\
	Step 1: (Local existence). For the existence, we use a fixed point argument.\\
	- If $p \in\left(0, \frac{4}{d+2\alpha+2}\right)$, fix $T, M>0$ and set
	$$
	\begin{aligned}
	X_{M}:=\left\{u \in L^{q}\left((-T, T) ; L_{\alpha}^{p+2}\left(\mathbb{R}^{d+1}_+\right)\right) \bigcap L^{\infty}\left((-T, T) ; L_{\alpha}^{2}\left(\mathbb{R}^{d+1}_+\right)\right):\right.& \\
	\|u\|_{L^{\infty}\left((-T, T) ; L_{\alpha}^{2}\left(\mathbb{R}^{d+1}_+\right)\right)}+\|u\|_{L^{q}\left((-T, T) ; L_{\alpha}^{p+2}\left(\mathbb{R}^{d+1}_+\right)\right)} \leq M &\}
	\end{aligned}
	$$
	where $(q, p+2)$ is sharp $\frac{d+2\alpha+2}{2}$-admissible pair. Note that by Theorem \ref{th4} and Corollary \ref{co3} , this space is never empty. Moreover, it is easily checked that $X_{M}$ is a complete metric space when equipped with the distance
	$$
	d(u, v)=\|u-v\|_{L^{\infty}\left((-T, T) ; L_{\alpha}^{2}\left(\mathbb{R}^{d+1}_+\right)\right)}+\|u-v\|_{L^{q}\left((-T, T) ; L_{\alpha}^{p+2}\left(\mathbb{R}^{d+1}_+\right)\right)}
	$$
	For simplify, we put
	$$
	\|v\|_{X_{M}}=\|v\|_{L^{\infty}\left((-T, T) ; L_{\alpha}^{2}\left(\mathbb{R}^{d}\right)\right)}+\|v\|_{L^{q}\left((-T, T) ; L_{\alpha}^{p+2}\left(\mathbb{R}^{d+1}_+\right)\right)}
	$$
	if $v \in X_{M}$. Take $g \in L_{\alpha}^{2}\left(\mathbb{R}^{d+1}_+\right)$. We wish to find conditions on $T$ and $M$ which imply that $\mathcal{H}_{\alpha}$, given by
	$$
	\mathcal{H}_{\alpha}(u)(t, \cdot):=\mathcal{I}_{\alpha}(t) g(\cdot)+\int_{0}^{t} \mathcal{I}_{\alpha}(t-s) F(u(s, \cdot)) d s
	$$
	is a strict contraction on $X_{M}$. By our nonlinearity assumption (\ref{eq}) and Theorem \ref{th4} the following estimate holds
	$$
	\left\|\mathcal{H}_{\alpha}(u)\right\|_{X_{M}} \leq C\left(\left\|\mathcal{I}_{\alpha}(\cdot) g\right\|_{X_{M}}+\|u\|_{L^{(p+1) q_{1}^{\prime}}}^{p+1}\left((-T, T) ; L_{\alpha}^{(p+1) r_{1}^{\prime}}\left(\mathbb{R}^{d+1}_+\right)\right)\right)
	$$
	with $\left(q_{1}, r_{1}\right)$ a sharp $\frac{d+2\alpha+2}{2}$-admissible couple.\\
	- If $p \in\left(0, \frac{4}{d+2\alpha+2}\right)$, we take $r_{1}=p+2$ and $\left(q_{1}=q, p+2\right)$ a sharp $\frac{d+2\alpha+2}{2}$ admissible pair such that $q>p+2$. Then applying Corollary \ref{co3} and Hölder's inequality in time we obtain
	\begin{align}\label{eq21}
	\left\|\mathcal{H}_{\alpha}(u)\right\|_{X_{M}} \leq C\|g\|_{{\alpha},{2}}+C T^{\frac{q-p-2}{q}}\|u\|_{L^{q}\left((-T, T) ; L_{\alpha}^{p+2}\left(\mathbb{R}^{d+1}_+\right)\right)}^{p+1}
	\end{align}
	Hence for every $u \in X_{M}$ one has
	$$
	\left\|\mathcal{H}_{\alpha}(u)\right\|_{X_{M}} \leq C\|g\|_{{\alpha},{2}}+C T^{\frac{q-p-2}{q}} M^{p+1}
	$$
	Choosing $M=2 C\|g\|_{L_{\alpha}^{2}\left(\mathbb{R}^{d+1}_+\right)}$, we see that if $T$ is sufficiently small (depending on $\|g\|_{L_{\alpha}^{2}\left(\mathbb{R}^{d+1}_+\right)}$ ) then $\mathcal{H}_{\alpha}(u) \in X_{M}$ for all $u \in X_{M}$. Moreover, arguing as above we obtain
	$$
	d\left(\mathcal{H}_{\alpha}(u), \mathcal{H}_{\alpha}(v)\right) \leq C T^{\frac{q-p-2}{q}} M^{p} d(u, v)
	$$
	for all $u, v \in X_{M}$. Thus $\mathcal{H}_{\alpha}$ is a contraction in $X_{M}$ provided $T$ is small enough,
	more precisely if $T \leq\left(\frac{1}{2 C M^{p}}\right)^{\frac{q}{q-p-2}}$. Hence $\mathcal{H}_{\alpha}$ has a fixed point $u$, which is the unique solution of (SW) in $X_{M}$, and there exist $T_{\max }, T_{\min } \in(0, \infty]$ such that $u$ belongs to
	$$
	C\left(\left(-T_{\min }, T_{\max }\right) ; L_{\alpha}^{2}\left(\mathbb{R}^{d+1}_+\right)\right) \bigcap L_{l o c}^{q}\left(\left(-T_{\min }, T_{\max }\right) ; L_{\alpha}^{r}\left(\mathbb{R}^{d+1}_+\right)\right)
	$$
	for the sharp $\frac{d+2\alpha+2}{2}$-admissible pair $(q, p+2)$, with\\
	$T_{\max }=\sup \{T>0$, there exists a solution of $(SW)$ on $[0, T]\}$\\
	$T_{\min }=\sup \{T>0$, there exists a solution of $(SW)$ on $[-T, 0]\}$\\
	Moreover, from Theorem \ref{th4} and by the argument we use to prove (\ref{eq21}), it is easy to see that $u \in L_{l o c}^{q_{1}}\left(\left(-T_{\min }, T_{\max }\right) ; L_{\alpha}^{r_{1}}\left(\mathbb{R}^{d+1}_+\right)\right)$ for every sharp $\frac{d+2\alpha+2}{2}$-admissible pair $\left(q_{1}, r_{1}\right)$\\
	- If $p=\frac{4}{d+2\alpha+2}$, let $g \in L_{\alpha}^{2}\left(\mathbb{R}^{d+1}_+\right) .$ Since $\mathcal{I}_{\alpha}(\cdot) g \in L^{p+2}\left(\mathbb{R} ; L_{\alpha}^{p+2}\left(\mathbb{R}^{d+1}_+\right)\right)$, by Corollary \ref{co3} , we have
	\begin{align}\label{eq22}
	\left\|\mathcal{I}_{\alpha}(.) g\right\|_{L^{p+2}\left((-T, T) ; L_{\alpha}^{p+2}\left(\mathbb{R}^{d+1}_+\right)\right)} \rightarrow 0 \quad \text { as } \quad T \downarrow 0
	\end{align}
	Therefore there exist $M, T>0$ such that
	\begin{align}\label{eq23}
	\left\|\mathcal{I}_{\alpha}(.) g\right\|_{L^{p+2}\left((-T, T) ; L_{\alpha}^{p+2}\left(\mathbb{R}^{d+1}_+\right)\right)}<M
\end{align}
	Let us consider the set
	$$
	X_{M}:=\left\{u \in L^{p+2}\left((-T, T) ; L_{\alpha}^{p+2}\left(\mathbb{R}^{d+1}_+\right)\right):\|u\|_{L^{p+2}\left((-T, T) ; L_{\alpha}^{p+2}\left(\mathbb{R}^{d+1}_+\right)\right)} \leq 2 M\right\}
	$$
	It is easily checked that $X_{M}$ is a complete metric space when equipped with the distance
	$$
	d(u, v)=\|u-v\|_{L^{p+2}\left((-T, T) ; L_{\alpha}^{p+2}\left(\mathbb{R}^{d+1}_+\right)\right)}
	$$
	As above, by Theorem \ref{th4} , the following estimate holds
	$$
	\begin{aligned}
	&\left\|\mathcal{H}_{\alpha}(u)\right\|_{L^{p+2}\left((-T, T) ; L_{\alpha}^{p+2}\left(\mathbb{R}^{d+1}_+\right)\right)} \\
	&\quad \leq C\left(\left\|\mathcal{I}_{\alpha}(\cdot) g\right\|_{L^{p+2}\left((-T, T) ; L_{\alpha}^{p+2}\left(\mathbb{R}^{d+1}_+\right)\right)}+\|u\|_{L^{(p+2)}\left((-T, T) ; L_{\alpha}^{(p+2)}\left(\mathbb{R}^{d+1}_+\right)\right)}^{p+1}\right)
	\end{aligned}
	$$
	where we have taken $q=q_{1}=r=r_{1}=p+2$. Hence, for every $u \in X_{M}$ :
	$$
	\left\|\mathcal{H}_{\alpha}(u)\right\|_{L^{p+2}\left((-T, T) ; L_{\alpha}^{p+2}\left(\mathbb{R}^{d+1}_+\right)\right)} \leq C\left\|\mathcal{I}_{\alpha}(\cdot) g\right\|_{L^{p+2}\left((-T, T) ; L_{\alpha}^{p+2}\left(\mathbb{R}^{d+1}_+\right)\right)}+C M^{p+1}
	$$
	From relations $(\ref{eq22})$ and $(\ref{eq23})$ above, we see that if $T$ is small enough, then we can choose $M$ such that $\mathcal{H}_{\alpha}(u)$ belongs to $X_{M}$ for all $u \in X_{M} .$ As above we prove also that $\mathcal{H}_{\alpha}$ is a contraction on the space $X_{M}$ provided $T$ is sufficiently small. Thus $\mathcal{H}_{\alpha}$ has a fixed point $u$, which is the unique solution of (SW) in $X_{M}$
	Moreover, from Theorem \ref{th4} it is easy to see that there exist $T_{\max }, T_{\min } \in(0, \infty]$ such that
	$$
	u \in C\left(\left(-T_{\min }, T_{\max }\right) ; L_{\alpha}^{2}\left(\mathbb{R}^{d+1}_+\right)\right) \bigcap L_{l o c}^{q_{1}}\left(\left(-T_{\min }, T_{\max }\right) ; L_{\alpha}^{r_{1}}\left(\mathbb{R}^{d+1}_+\right)\right)
	$$
	for every sharp $\frac{d+2\alpha+2}{2}$-admissible pair $\left(q_{1}, r_{1}\right)$.\\
	
	Step 2: (Uniqueness). We first note that the uniqueness is a local property, so that we need only to establish it on possibly small intervals. To see this, we argue for positive times, the case for negative times being the same. Suppose that $u_{1}, u_{2} \in C\left([0, T] ; L_{\alpha}^{2}\left(\mathbb{R}^{d+1}_+\right)\right) \bigcap L_{l o c}^{q}\left((0, T) ; L_{\alpha}^{r}\left(\mathbb{R}^{d+1}_+\right)\right)$ are two solutions of (SW) and assume that $u_{1}(t) \neq u_{2}(t)$ for some $t \in[0, T] .$ Let $t_{0}=\inf \left\{t \in[0, T], u_{1}(t) \neq\right.$ $\left.u_{2}(t)\right\}$. Since both $u_{1}$ and $u_{2}$ are continuous into $L_{\alpha}^{2}\left(\mathbb{R}^{d+1}_+\right)$, this definition makes sense and $u_{1}\left(t_{0}\right)=u_{2}\left(t_{0}\right)=\chi .$ Moreover, the curves $U_{1}(t)=u_{1}\left(t+t_{0}\right)$ and $U_{2}(t)=u_{2}\left(t+t_{0}\right)$ both satisfy the equation $w=\mathcal{I}_{\alpha}(\cdot) \chi+\Phi_{\alpha}(F(w))$ on $\left[0, T-t_{0}\right]$ As above we apply Theorem \ref{th4} and the argument of proof of (\ref{eq21}), to obtain that for all $t \in\left[t_{0}, T\right]$
	$$
	\begin{aligned}
	&\left\|u_{1}-u_{2}\right\|_{L^{q}\left(\left(t_{0}, t\right) ; L_{\alpha}^{p+2}\left(\mathbb{R}^{d+1}_+\right)\right)} \\
	&\quad \leq C\left(t-t_{0}\right)^{\frac{4-(d+2\alpha+2) p}{4}} \sum_{i=1}^{2}\left\|u_{i}\right\|_{L^{q}\left(\left(t_{0}, t\right) ; L_{\alpha}^{p+2}\left(\mathbb{R}^{d+1}_+\right)\right)}^{p}\left\|u_{1}-u_{2}\right\|_{L^{q}\left(\left(t_{0}, t\right) ; L_{\alpha}^{p+2}\left(\mathbb{R}^{d+1}_+\right)\right)}
	\end{aligned}
	$$
	where $\left(q=\frac{4(p+2)}{p(d+2\alpha+2)}, p+2\right)$ is a sharp $\frac{d+2\alpha+2}{2}$-admissible pair. For $t>t_{0}$, but sufficiently close to $t_{0}$, it follows that
	$$
	C\left(t-t_{0}\right)^{\frac{4-(d+2\alpha+2) p}{4}} \sum_{i=1}^{2}\left\|u_{i}\right\|_{L^{q}\left(\left(t_{0}, t\right) ; L_{\alpha}^{p+2}\left(\mathbb{R}^{d+1}_+\right)\right)}^{p}<1
	$$
	and so that $$\left\|u_{1}-u_{2}\right\|_{L^{q}\left(\left(t_{0}, t\right) ; L_{\alpha}^{p+2}\left(\mathbb{R}^{d+1}_+\right)\right)}=0.$$ This contradicts the choice of $t_{0}$, and thus proves that $u_{1}(t)=u_{2}(t)$ for all $t \in[0, T]$.\\
	
	Step 3: (Global existence). As above, by Theorem \ref{th4} the following Strichartz estimate holds
	\begin{align}\label{eq24}
	\left\|\mathcal{H}_{\alpha}(u)\right\|_{X} \leq C\left(\|g\|_{L_{\alpha}^{2}\left(\mathbb{R}^{d+1}_+\right)}+\|F(u)\|_{L^{q^{\prime}}\left(\mathbb{R} ; L_{\alpha}^{r^{\prime}}\left(\mathbb{R}^{d+1}_+\right)\right)}\right)
	\end{align}
	with $(q, r=p+2)$ a sharp $\frac{d+2\alpha+2}{2}$-admissible and $X=C\left(\mathbb{R} ; L_{\alpha}^{2}\left(\mathbb{R}^{d+1}_+\right)\right) \bigcap L^{q}\left(\mathbb{R} ; L_{\alpha}^{r}\left(\mathbb{R}^{d+1}_+\right)\right)$ the Banach space with norm
	$$
	\|v\|_{X}=\|v\|_{L^{q}\left(\mathbb{R} ; L_{\alpha}^{r}\left(\mathbb{R}^{d+1}_+\right)\right)}+\|v\|_{C\left(\mathbb{R} ; L_{\alpha}^{2}\left(\mathbb{R}^{d+1}_+\right)\right)}
	$$
	By our nonlinearity assumption (\ref{eq}) and Hölder inequality we have
	$$
	\left\|\mathcal{H}_{\alpha}(u)\right\|_{X} \leq C\left(\|g\|_{L_{\alpha}^{2}\left(\mathbb{R}^{d+1}_+\right)}+\|u\|_{L^{q}\left(\mathbb{R} ; L_{\alpha}^{r}\left(\mathbb{R}^{d+1}_+\right)\right)}^{p+1}\right)
	$$
	As above, we have proved that $\mathcal{H}_{\alpha}$ maps the Banach space $X$ into itself, and moreover the ball $X_{M}$ into itself, provided $M$ and $\|g\|_{L_{\alpha}^{2}\left(\mathbb{R}^{d+1}_+\right)}$ are small enough, where
	$$
	X_{M}=\left\{u \in X:\|u\|_{X}<M\right\}
	$$
	We assume now that $u_{i} \in X$ is such that
	$$
	\left\|u_{i}\right\|_{X}<M
	$$
	with $M$ small enough, and also that $\|g\|_{L_{\alpha}^{2}\left(\mathbb{R}^{d+1}_+\right)}<\delta$. By (5.75) we note that
	$$
	\left\|\mathcal{H}_{\alpha}(u)\right\|_{X} \leq C \delta+C M^{p+1}<M
	$$
	provided $M, \delta$ are such that $C M^{p}<\frac{1}{2}$ and $C \delta<\frac{M}{2}$. We have also
	$$
	\begin{aligned}
	\| \mathcal{H}_{\alpha}\left(u_{1}\right)-& \mathcal{H}_{\alpha}\left(u_{2}\right)\left\|_{X} \leq C\right\| F\left(u_{1}\right)-F\left(u_{2}\right) \|_{L^{q^{\prime}}\left(\mathbb{R} ; L_{\alpha}^{\prime}\left(\mathbb{R}^{d+1}_+\right)\right)} \\
	& \leq C\left\|u_{1}-u_{2}\right\|_{L^{q}\left(\mathbb{R}; L_{\alpha}^{r}\left(\mathbb{R}^{d+1}_+\right)\right)}\left(\left\|u_{1}\right\|_{L^{q}\left(\mathbb{R} ; L_{\alpha}^{r}\left(\mathbb{R}^{d+1}_+\right)\right)}^{p}+\left\|u_{2}\right\|_{L^{q}\left(\mathbb{R} ; L_{\alpha}^{r}\left(\mathbb{R}^{d+1}_+\right)\right)}^{p}\right) \\
	& \leq\left\|u_{1}-u_{2}\right\|_{X_{M}} 2 C M^{p} \leq \frac{1}{2}\left\|u_{1}-u_{2}\right\|_{X_{M}}
	\end{aligned}
	$$
	provided $M$ is so small that $2 C M^{p}<\frac{1}{2} .$ Thus, if initial data are small enough i.e. $\|g\|_{L_{\alpha}^{2}\left(\mathbb{R}^{d+1}_+\right)}<\delta$, then the map $\mathcal{H}_{\alpha}$ is a contraction and this implies that there exists a unique solution $u(t, x)$ of the Cauchy problem (SW) such that $u(t, x) \in L^{q}\left(\mathbb{R} ; L_{\alpha}^{r}\left(\mathbb{R}^{d+1}_+\right)\right)$ with a couple $(q, r)$ which is sharp $\frac{d+2\alpha+2}{2}$-admissible pair when $1<p \leq \frac{4}{d+2\alpha+2}$. As observed above one can see easily that this is the unique solution in $u(t, x) \in C\left(R, L_{\alpha}^{2}\left(\mathbb{R}^{d+1}_+\right)\right)$ with small initial data in $L_{\alpha}^{2}$. Thus we have proved the global existence as claimed.
	\end{proof}
\begin{propo} We assume that $F$ is as in Theorem \ref{th5}. If $g \in L_{\alpha}^{2}\left(\mathbb{R}^{d+1}_+\right)$ and if $u$ is the maximal solution of $(SW)$, then we have:\\
	
i) If $p \in\left(0, \frac{4}{d+2\alpha+2}\right)$ and $T_{\max }<\infty$ (respectively, $T_{\min }<\infty$ ), then $$\|u(t)\|_{{\alpha},{2}}\rightarrow \infty  \text { as } t \uparrow T_{\max } (\text {respectively, as } .t \downarrow T_{\min })$$.
ii) If $p=\frac{4}{d+2\alpha+2}$ and $T_{\max }<\infty$ (respectively, $T_{\min }<\infty$ ), then $$\|u\|_{L^{q}\left(\left(0, T_{\max }\right) ; L_{\alpha}^{r}\left(\mathbb{R}^{d+1}_+\right)\right)}=\infty\left(\right.\text{respectively, }\left.\|u\|_{L^{q}\left(\left(-T_{\min }, 0\right) ; L_{\alpha}^{r}\left(\mathbb{R}^{d+1}_+\right)\right)}=\infty\right)$$ for every sharp $\frac{d+2\alpha+2}{2}$-admissible pair $(q, r)$ with $r \geq p+2$.
\end{propo}
\begin{proof}
 i) If $p \in\left(0, \frac{4}{d+2\alpha+2}\right)$, it follows from Step 1 of the proof of Theorem \ref{th5} and the uniqueness property that $$T_{\max }-t \geq\left(\frac{1}{4 C^{2}\|u(t, \cdot)\|_{{\alpha},{2}}^{p}}\right)^{\frac{q}{q-p-2}}$$.
	Suppose now that $T_{\max }<\infty$, then 
	$$\|u(t)\|_{{\alpha},{2}} \geq\left(\frac{1}{4 C^{2}\left(T_{\max }-t\right)^{\frac{q-p-2}{q}}}\right)^{\frac{1}{p}} .$$
	 As $q>p+2$, we obtain $\|u(t)\|_{{\alpha},{2}} \rightarrow \infty$ as $t \uparrow T_{\max }$. One shows by the same argument that if $T_{\min }<\infty$, then $\|u(t)\|_{{\alpha},{2}} \rightarrow \infty$ as $t \downarrow T_{\min }$.\\
	ii) If $p=\frac{4}{d+2\alpha+2}$, we show the blowup alternative by contradiction. Suppose that $T_{\max }<\infty$ and that $\|u\|_{L^{(p+2)}\left(\left(0, T_{\max }\right) ; L_{\alpha}^{(p+2)}\left(\mathbb{R}^{d+1}_+\right)\right)}<\infty .$ Let $0 \leq t \leq$ $t+\tau<T_{\max } .$ It follow that
	$$
	\mathcal{I}_{\alpha}(\tau) u(t, \cdot)=u(t+\tau, \cdot)-\int_{0}^{\tau} \mathcal{I}_{\alpha}(\tau-s) F(u(t+s, \cdot)) d s
	$$
	By Theorem \ref{th4} we deduce that there exists $C$ such that
	$$
	\begin{aligned}
	\left\|\mathcal{I}_{\alpha}(\cdot) u(t)\right\|_{L^{(p+2)}\left(\left(0, T_{\max }-t\right) ; L_{\alpha}^{(p+2)}\left(\mathbb{R}^{d+1}_+\right)\right)} & \leq\|u\|_{L^{(p+2)}\left(\left(t, T_{\max }\right) ; L_{\alpha}^{(p+2)}\left(\mathbb{R}^{d+1}_+\right)\right)} \\
	&+C\|u\|_{L^{(p+2)}\left(\left(t, T_{\max }\right) ; L_{\alpha}^{(p+2)}\left(\mathbb{R}^{d+1}_+\right)\right)}^{p+1}
	\end{aligned}
	$$
	Therefore for $t$ close enough to $T_{\max }$,
	$$
	\left\|\mathcal{I}_{\alpha}(\cdot) u(t)\right\|_{L^{(p+2)}\left(\left(0, T_{\max }-t\right) ; L_{\alpha}^{(p+2)}\left(\mathbb{R}^{d+1}_+\right)\right)} \leq \frac{M}{2}
	$$
	By Step 1 in the proof of Theorem $\ref{th5}, u$ can be extended beyong $T_{\max }$, which is a contradiction. This shows that $\|u\|_{L^{(p+2)}\left(\left(0, T_{\max }\right) ; L_{\alpha}^{(p+2)}\left(\mathbb{R}^{d+1}_+\right)\right)}=\infty .$ Let now $(q, r)$ be a sharp $\frac{d+2\alpha+2}{2}$-admissible pair such that $r>p+2 .$ It follows from Hölder's inequality that for any $T<T_{\max }$,
	$$
	\|u\|_{L^{(p+2)}\left((0, T) ; L_{\alpha}^{(p+2)}\left(\mathbb{R}^{d+1}_+\right)\right)} \leq\|u\|_{L^{\infty}\left((0, T) ; L_{\alpha}^{2}\left(\mathbb{R}^{d+1}_+\right)\right)}^{\lambda}\|u\|_{L^{q}\left((0, T) ; L_{\alpha}^{r}\left(\mathbb{R}^{d+1}_+\right)\right)}^{1-\lambda}
	$$
	with $\lambda=\frac{2(r-p-2)}{(p+2)(r-2)} .$ Letting $T \uparrow T_{\max }$, we obtain $\|u\|_{L^{q}\left(\left(0, T_{\max }\right) ; L_{\alpha}^{r}\left(\mathbb{R}^{d+1}_+\right)\right)}=\infty .$
	One shows by the same argument that if $T_{\min }<\infty$, then $\|u\|_{L^{q}\left(\left(-T_{\min }, 0\right) ; L_{\alpha}^{r}\left(\mathbb{R}^{d+1}_+\right)\right)}=$ $\infty$
\end{proof}

\end{document}